\documentclass[10pt]{article}
\font\smallit=cmti10

\usepackage{amssymb,latexsym,amsmath,epsfig,amsthm} 
\newtheorem{theorem}{\textbf{Theorem}}
\newtheorem{lemma}{\textbf{Lemma}}
\newtheorem{prop}{\textbf{Proposition}}
\newtheorem{corollary}{\textbf{Corollary}}
\theoremstyle{definition}
\newtheorem{definition}{\textbf{Definition}}
\newtheorem{remark}{\textbf{Remark}}

\makeatletter

\renewcommand\section{\@startsection {section}{1}{\z@}
{-30pt \@plus -1ex \@minus -.2ex}
{2.3ex \@plus.2ex}
{\normalfont\normalsize\bfseries}}

\renewcommand\subsection{\@startsection{subsection}{2}{\z@}
{-3.25ex\@plus -1ex \@minus -.2ex}
{1.5ex \@plus .2ex}
{\normalfont\normalsize\bfseries}}

\renewcommand{\@seccntformat}[1]{\csname the#1\endcsname. }

\makeatother

\begin{document}

\begin{center}
\uppercase{\bf Robin's theorem, primes, and a new elementary reformulation of the Riemann Hypothesis}
\vskip 20pt
{\bf Geoffrey Caveney
}\\
{\smallit 7455 North Greenview \#426,
Chicago, IL 60626, USA}\\
{\tt rokirovka@gmail.com}\\
\vskip 10pt
{\bf Jean-Louis Nicolas
}\\
{\smallit Universit\'e de Lyon; CNRS; Universit\'e Lyon 1;\\
Institut Camille Jordan, Math\'ematiques,\\
21 Avenue Claude Bernard, F-69622 Villeurbanne cedex, France}\\
{\tt nicolas@math.univ-lyon1.fr}\\
\vskip 10pt
{\bf Jonathan Sondow
}\\
{\smallit 209 West 97th Street \#6F,
New York, NY 10025, USA}\\
{\tt jsondow@alumni.princeton.edu}\\
\end{center}
\vskip 30pt
\vskip 30pt


\centerline{\bf Abstract}
\noindent Let$$G(n)=\frac{\sigma(n)}{n \log\log n} \qquad(n>1),$$where $\sigma(n)$ is the sum of the divisors of $n$. We prove that the Riemann Hypothesis is true if and only if $4$ is the only composite number~$N$ satisfying$$G(N) \ge \max\left(G(N/p), G(aN)\right),$$for all prime factors $p$ of $N$ and each positive integer~$a$. The proof uses Robin's and Gronwall's theorems on $G(n)$. An alternate proof of one step depends on two properties of superabundant numbers proved using Alaoglu and Erd\H{o}s's results.

\pagestyle{myheadings} 
\thispagestyle{empty} 
\baselineskip=12.875pt 
\setcounter{page}{1} 
\vskip 30pt

\section{Introduction} \label{SEC: intro}

The \emph{sum-of-divisors function} $\sigma$ is defined by
$$\sigma(n):=\sum_{d \mid n}d.$$
For example, $\sigma(4)=7$ and $\sigma(pn)=(p+1)\sigma(n)$, if $p$ is a prime not dividing~$n$.

In 1913, the Swedish mathematician Thomas Gronwall \cite{gronwall} found the maximal order of~$\sigma$.

\begin{theorem}[Gronwall] \label{THM: gronwall}
The function
$$G(n):=\frac{\sigma(n)}{n \log\log n} \qquad(n>1)$$
satisfies
$$\limsup_{n\to\infty} G(n) =  e^\gamma = 1.78107\dotso,$$
where $\gamma$ is the Euler-Mascheroni constant.
\end{theorem}

Here $\gamma$ is defined as the limit
$$\gamma := \lim_{n\to\infty} (H_n - \log n) = 0.57721\dotsc,$$
where $H_n$ denotes the \emph{$n$th harmonic number}
$$H_n := \sum_{j=1}^n \frac1j = 1+\frac12+\frac13+\dotsb+\frac1n.$$

Gronwall's proof uses Mertens's theorem \cite[Theorem~429]{hw}, which says that if $p$~denotes a prime, then
$$\lim_{x\to\infty} \frac{1}{{\log x}} \prod_{p\le x}\left(1-\frac{1}{p}\right)^{-1}= e^{\gamma}.$$

Since $\sigma(n) >n$ for all $n>1$, Gronwall's theorem ``shows that the order of $\sigma(n)$ is always `very nearly' $n$'' (Hardy and Wright \cite[p.~350]{hw}).

In 1915, the Indian mathematical genius Srinivasa Ramanujan proved an asymptotic inequality for Gronwall's function $G$, assuming the Riemann Hypothesis (RH). (Ramanujan's result was not published until much later \cite{ramanujan97}; for the interesting reasons, see \cite[pp.~119--121]{ramanujan97} and \cite[pp.~537--538]{lagarias}.)

\begin{theorem}[Ramanujan] \label{THM: ramaujan}
If the Riemann Hypothesis is true, then
\begin{equation*}
G(n) < e^\gamma \qquad (n \gg 1).
\end{equation*}
\end{theorem}

Here $n\gg1$ means for all sufficiently large $n$.

In 1984, the French mathematician Guy Robin \cite{robin} proved that a stronger statement about the function $G$ is \emph{equivalent} to  the RH.

\begin{theorem}[Robin] \label{THM: robin}
The Riemann Hypothesis is true if and only if
\begin{equation}
G(n) < e^\gamma \qquad (n > 5040). \label{EQ: robin}
\end{equation}
\end{theorem}

The condition~\eqref{EQ: robin} is called \emph{Robin's inequality}. Table~\ref{TABLE: S and G(m)} gives the known numbers~$r$ for which the reverse inequality $G(r) \ge  e^\gamma$ holds, together with the value of $G(r)$ (truncated).

Robin's statement is elementary, and his theorem is beautiful and elegant, and is certainly quite an achievement.

\begin{table}[htdp]
\begin{center}
\begin{tabular}{rclrcccc}
\hline
$r\ $ &  SA  &  Factorization & $\sigma(r)\, $ & $\sigma(r)/r$ & $G(r)$ & $p(r)$ & $G(11r)$\\
\hline
$3$  &  &   3 & 4& 1.333 & $14.177\ \ $ & & 1.161\\
$4$  &   \checkmark &  $2^2$ & 7 & 1.750 & 5.357 & & 1.434 \\
$5$   &   & 5  & 6& 1.200 & 2.521 & & 0.943 \\
$6$   & \checkmark &   $2\cdot3$ &  12 & 2.000 & 3.429 & 2 & 1.522 \\
$8$ &     & $2^3$ &  15 & 1.875 & 2.561 & 2 & 1.364 \\
$9$  & & $3^2$ & 13& 1.444 & 1.834 & 3 & 1.033 \\
$10$  &  &   $2\cdot5$ &18 & 1.800 & 2.158 & 2 & 1.268\\
$12$ &  \checkmark &  $2^2\cdot3$ & 28 & 2.333 & 2.563 & 2 & 1.605 \\
$16$  &  &  $2^4$ &31 & 1.937 & 1.899 & 2 & 1.286 \\
$18$  &  &   $2\cdot3^2$  & 39 & 2.166 & 2.041 & 3 & 1.419 \\
$20$  &  &   $2^2\cdot5$ & 42 & 2.100 & 1.913 & 5 & 1.359 \\
$24$  & \checkmark &   $2^3\cdot3$&  60 & 2.500 & 2.162 & 3 & 1.587 \\
$30$ &   &   $2\cdot3\cdot5$  & 72 & 2.400 & 1.960 & 3 & 1.489 \\
$36$  & \checkmark &  $2^2\cdot3^2$ & 91 & 2.527 & 1.980 & 2 & 1.541 \\
$48$ &  \checkmark &   $2^4\cdot3$ & 124 & 2.583 & 1.908 & 3 & 1.535 \\
$60$ &  \checkmark &   $2^2\cdot3\cdot5$ &  168 & 2.800 & 1.986 & 5 & 1.632 \\
$72$ &   & $2^3\cdot3^2$  & 195 & 2.708 & 1.863 & 3 & 1.556 \\
$84$  &    & $2^2\cdot3\cdot7$ & 224 & 2.666 & 1.791 & 7 & 1.514 \\
$120$ &  \checkmark &  $2^3\cdot3\cdot5$  &  360 & 3.000 & 1.915 & 2 & 1.659\\
$180$  & \checkmark  &  $2^2\cdot3^2\cdot5$ &  546 & 3.033 & 1.841 & 5 & 1.632 \\
$240$  & \checkmark &   $2^4\cdot3\cdot5$&  744 & 3.100 & 1.822 & 5 & 1.638 \\
$360$   & \checkmark  &  $2^3\cdot3^2\cdot5$&  1170 & 3.250 & 1.833 & 5 & 1.676 \\
$720$  & \checkmark &   $2^4\cdot3^2\cdot5$ &  2418 & 3.358 & 1.782 & 3 & 1.669 \\
$840$ &  \checkmark &   $2^3\cdot3\cdot5\cdot7$&  2880 & 3.428 & 1.797 & 7 & 1.691 \\
$2520$  & \checkmark &   $2^3\cdot3^2\cdot5\cdot7$ &  9360 & 3.714 & 1.804 & 7 & 1.742 \\
$5040$  & \checkmark &   $2^4\cdot3^2\cdot5\cdot7$&  19344 & 3.838 & 1.790 & 2 & 1.751 \\
\hline
\end{tabular}
\end{center}
\caption{The set $R$ of all known numbers $r$ for which $G(r) \ge  e^\gamma$. (Section~\ref{SEC: intro} defines SA, $\sigma(r)$, and $G(r)$; Section~\ref{SEC: lemmas} defines $p(r)$.)}
\label{TABLE: S and G(m)}
\end{table}%

In \cite{robin} Robin also proved, unconditionally, that
\begin{equation}
G(n) < e^\gamma + \frac{0.6483}{(\log\log n)^2} \qquad (n > 1).\label{EQ: bound}
\end{equation}
This refines the inequality $\limsup_{n\to\infty} G(n) \le  e^\gamma$ from Gronwall's theorem.

In 2002, the American mathematician Jeffrey Lagarias \cite{lagarias} used Robin's theorem to give another elementary reformulation of the RH.

\begin{theorem}[Lagarias] \label{THM: lagarias}
The Riemann Hypothesis is true if and only if
$$\sigma(n) < H_n + \exp(H_n)\log(H_n) \qquad (n>1).$$
\end{theorem}

Lagarias's theorem is also a beautiful, elegant, and remarkable achievement. It improves upon Robin's statement in that it does not require the condition $n > 5040$, which appears arbitrary. It also differs from Robin's statement in that it relies explicitly on the harmonic numbers $H_n$ rather than on the constant $\gamma$.

Lagarias \cite{lagarias2} also proved, unconditionally, that
$$\sigma(n) < H_n + 2\exp(H_n)\log(H_n) \qquad (n>1).$$

The present note uses Robin's results to derive another reformulation of the RH. Before stating it, we give a definition and an example.

\begin{definition} \label{DEF: exceptional}
A positive integer $N$ is \emph{extraordinary} if $N$ is composite and satisfies

(i). $G(N) \ge G(N/p)$ for all prime factors $p$ of $N$, and

(ii). $G(N) \ge G(aN)$ for all multiples $aN$ of $N$.
\end{definition}

The smallest extraordinary number is $N=4$. To show this, we first compute $G(4)=5.357\dotso$. Then as $G(2)<0$, condition (i) holds, and since Robin's unconditional bound \eqref{EQ: bound} implies
\begin{equation*}
G(n) < e^\gamma + \frac{0.6483}{(\log\log 5)^2} = 4.643\dotso < G(4) \qquad(n\ge5),
\end{equation*}
condition (ii) holds a fortiori.

No other extraordinary number is known, for a good reason.

\begin{theorem} \label{THM: Caveney}
The Riemann Hypothesis is true if and only if $4$ is the only extraordinary number.
\end{theorem}

This statement is elementary and involves prime numbers (via the definition of an extraordinary number) but not the constant~$\gamma$ or the harmonic numbers  $H_n$, which are difficult to calculate and work with for large values of $n$. On the other hand, to disprove the RH using Robin's or Lagarias's statement would require only a \emph{calculation} on a certain number $n$, while using ours would require a \emph{proof} for a certain number $N$.

Here is a near miss. One can check that the number
\begin{equation}
\nu:=183783600 = 2^4\cdot3^3\cdot5^2\cdot7\cdot11\cdot13\cdot17 \label{EQ: near}
\end{equation}
satisfies condition (i), that is, $G(\nu) \ge G(\nu/p)$ for $p=2,3,5,7,11,13,17$. However, $\nu$~is not extraordinary, because $G(\nu) < G(19\nu)$. Thus $183783600$ is not quite a counterexample to the RH!

In \cite[Section 59]{ramanujan97} Ramanujan introduced the notion of a ``generalized highly composite number.'' The terminology was changed to ``superabundant number'' by the Canadian-American mathematician Leonidas Alaoglu and the Hungarian mathematician Paul Erd\H{o}s \cite{ae}.

\begin{definition}[Ramanujan and Alaoglu-Erd\H{o}s] \label{DEF: super}
A positive integer $s$ is \emph{superabundant} (\emph{SA}) if$$\frac{\sigma(n)}{n} < \frac{\sigma(s)}{s} \qquad(0<n<s).$$
\end{definition}

For example, the numbers $1,2$, and $4$ are SA, but $3$ is not SA, because
$$\frac{\sigma(1)}{1} = 1 < \frac{\sigma(3)}{3} = \frac43 < \frac{\sigma(2)}{2} = \frac32 < \frac{\sigma(4)}{4} = \frac74.$$
For lists of SA numbers, see the links at \cite[Sequence~A004394]{oeis} and the last table in \cite{ae}. The known SA numbers $s$ for which $G(s) \ge e^\gamma$ are indicated in the ``SA'' column of Table~\ref{TABLE: S and G(m)}. Properties of SA numbers are given in \cite{ae,briggs,lagarias,ramanujan97}, Proposition~\ref{PROP: SA1}, and Section~\ref{SEC: appendix}.

As $\sigma(n)/n=G(n)\log\log n$, Gronwall's theorem yields $\limsup_{n\to\infty} \sigma(n)/n = \infty$, implying \emph{there exist infinitely many SA numbers}.

Let us compare Definition~\ref{DEF: super} with condition (i) in Definition~\ref{DEF: exceptional}. If $n<s$, then $\sigma(n)/n < \sigma(s)/s$ is a weaker inequality than $G(n)<G(s)$. On the other hand, condition~(i) only requires $G(n)\le G(N)$ for factors $n=N/p$, while Definition~\ref{DEF: super} requires $\sigma(n)/n < \sigma(s)/s$ for all $n<s$. In particular, \emph{the near miss \eqref{EQ: near} is the smallest SA number greater than $4$ that satisfies}~(i). For more on (i), see Section~\ref{SEC: GA}.

Amir Akbary and Zachary Friggstad \cite{af} observed that, ``\emph{If there is any counterexample to Robin's inequality, then the least such counterexample is a superabundant number.}'' Combined with Robin's theorem, their result implies \emph{the RH is true if and only if $G(s) < e^\gamma$ for all SA numbers $s>5040$.}

Here is an analog for extraordinary numbers of Akbary and Friggstad's observation on SA numbers.

\begin{corollary} \label{COR: N}
If there is any counterexample to Robin's inequality, then the maximum $M:=\max\{G(n):n > 5040\}$ exists
and the least number $N>5040$ with $G(N)=M$ is extraordinary.
\end{corollary}

Using Gronwall's theorem and results of Alaoglu and Erd\H{o}s, we prove two properties of SA numbers.

\begin{prop} \label{PROP: SA1}
Let $S$ denote the set of superabundant numbers.

\noindent\emph{SA1.} We have
\begin{equation*}
\limsup_{s\in S} G(s) =  e^\gamma. 
\end{equation*}
\noindent\emph{SA2.} For any fixed positive integer $n_0$, every sufficiently large number $s\in S$ is a multiple of~$n_0$.
\end{prop}

The rest of the paper is organized as follows. The next section contains three lemmas about the function $G$; an alternate proof of the first uses Proposition~\ref{PROP: SA1}. The lemmas are used in the proof of Theorem~\ref{THM: Caveney} and Corollary~\ref{COR: N}, which is in Section~\ref{SEC: proof}. Proposition~\ref{PROP: SA1} is proved in Section~\ref{SEC: appendix}. Section~\ref{SEC: GA} gives some first results about numbers satisfying condition~(i) in Definition~\ref{DEF: exceptional}.

We intend to return to the last subject in another paper \cite{cns}, in which we will also study numbers satisfying condition~(ii).

\section{Three lemmas on the function $G$} \label{SEC: lemmas}

The proof of Theorem \ref{THM: Caveney} requires three lemmas. Their proofs are unconditional.

The first lemma generalizes Gronwall's theorem (the case $n_0=1$).

\begin{lemma} \label{LEM: G(an)}
If $n_0$ is any fixed positive integer, then
$\displaystyle\limsup_{a\to\infty} G(an_0) =  e^\gamma.$
\end{lemma}
We give two proofs.
\begin{proof}[Proof 1]
Theorem~\ref{THM: gronwall} implies $\limsup_{a\to\infty} G(an_0) \le e^\gamma$. The reverse inequality can be proved by adapting that part of the proof of Theorem~\ref{THM: gronwall} in \cite[Section~22.9]{hw}. Details are omitted.
\end{proof}
\begin{proof}[Proof 2]
The lemma follows immediately from Proposition~\ref{PROP: SA1}.
\end{proof}

The remaining two lemmas give properties of the set $R$ of all known numbers $r$ for which $G(r) \ge  e^\gamma.$

\begin{lemma} \label{LEM: G(n/p)}
Let $R$ denote the set
$$R:=\{r\le5040 : G(r) \ge  e^\gamma\}.$$
If $r \in R$ and $r>5$, then $G(r) < G(r/p)$, for some prime factor $p$ of $r$.
\end{lemma}
\begin{proof}
The numbers $r\in R$ and the values $G(r)$ are computed in Table~\ref{TABLE: S and G(m)}. Assuming $G(r) < G(r/p)$ for some prime factor $p$ of $r$, denote the smallest such prime by
$$p(r):= \min\{\text{prime}\ p\mid r : G(r/p) > G(r)\}.$$
Whenever $5<r\in R$, a value of $p(r)$ is exhibited in the ``$p(r)$'' column of Table~\ref{TABLE: S and G(m)}. This proves the lemma.
\end{proof}

\begin{lemma} \label{LEM: G(mp)}
If $r \in R$ and $p  \ge 11$ is prime, then $G(pr) < e^\gamma$.
\end{lemma}
\begin{proof}
Note that if $p>q$ are odd primes not dividing a number $n$, then
$$G(pn)=\frac{\sigma(pn)}{pn \log\log pn}=\frac{p+1}{p}\frac{\sigma(n)}{n\log\log pn} <\frac{q+1}{q} \frac{\sigma(n)}{n\log\log qn}= G(qn).$$
Also, Table~\ref{TABLE: S and G(m)} shows that no prime $p  \ge 11$ divides any number $r \in R$, and that $G(11r)<1.76$ for all $r \in R$. As $1.76< e^\gamma$, we obtain $G(pr) \le G(11r) < e^\gamma$.
\end{proof}
Note that the inequality $G(pn) < G(qn)$ and its proof remain valid for \emph{all} primes $p>q$ not dividing~$n$, if $n>1$, since then $\log \log qn \neq \log \log2<0$ when $q=2$.

\section{Proof of Theorem \ref{THM: Caveney} and Corollary~\ref{COR: N}} \label{SEC: proof}

We can now prove that our statement is equivalent to the RH.

\begin{proof}[Proof of Theorem \ref{THM: Caveney} and Corollary \ref{COR: N}]
Assume $N\neq4$ is an extraordinary number. Then condition~(ii) and Lemma~\ref{LEM: G(an)} imply $G(N) \ge e^\gamma$. Thus if $N \le 5040$, then $N\in R$, but now since $N\neq4$ is composite we have $N>5$, and Lemma~\ref{LEM: G(n/p)} contradicts condition~(i). Hence $N > 5040$, and by Theorem~\ref{THM: robin} the RH is false.

Conversely, suppose the RH is false. Then from Theorems \ref{THM: gronwall} and~\ref{THM: robin} we infer that the maximum
\begin{equation}
M:=\max\{G(n):n > 5040\} \label{EQ: max}
\end{equation}
exists and that $M \ge e^\gamma$. Set
\begin{equation}
N:=\min\{n>5040:G(n)=M\} \label{EQ: min}
\end{equation}
and note that $G(N) = M \ge e^\gamma$. We show that $N$ is an extraordinary number.

First of all, $N$ is composite, because if $N$ is prime, then $\sigma(N)=1+N$ and $N>5040$ imply $G(N) < 5041/(5040 \log\log5040)=0.46672\dotsc$, contradicting $G(N) \ge e^\gamma$.

Since \eqref{EQ: max} and \eqref{EQ: min} imply $G(N) \ge G(n)$ for all $n \ge N$, condition (ii) holds. To see that (i)~also holds, let prime $p$ divide $N$ and set $r:=N/p$. In the case $r> 5040$, as $r<N$ the minimality of $N$ implies $G(N) > G(r)$. Now consider the case $r\le 5040$. By computation, $G(n)<e^\gamma$ if $5041\le n \le 35280$, so that $N > 35280=7\cdot5040$ and hence $p \ge 11$. Now if $G(r) \ge e^\gamma$, implying $r \in R$, then Lemma~\ref{LEM: G(mp)} yields $e^\gamma > G(pr) = G(N)$, contradicting $G(N) \ge e^\gamma$. Hence $G(r) < e^\gamma \le G(N)$. Thus in both cases $G(N) > G(r)=G(N/p)$, and so (i)~holds. Therefore, $N\neq4$ is extraordinary. This proves both the theorem and the corollary.
\end{proof}

\begin{remark}
The proof shows that Theorem~\ref{THM: Caveney} and Corollary \ref{COR: N} remain valid if we replace the inequality in Definition~\ref{DEF: exceptional} (i) with the strict inequality $G(N) > G(N/p)$.
\end{remark}

\section{Proof of Proposition~\ref{PROP: SA1}} \label{SEC: appendix}

We prove the two parts of Proposition~\ref{PROP: SA1} separately.

\begin{proof}[Proof of SA1]
It suffices to construct a sequence $s_1,s_2,\dotso\to\infty$ with $s_k\in S$ and $\limsup_{k\to\infty} G(s_k) \ge  e^\gamma$. By Theorem~\ref{THM: gronwall}, there exist positive integers $\nu_1<\nu_2<\dotsb$ with $\lim_{k\to\infty} G(\nu_k) =  e^\gamma$. If $\nu_k\in S$, set $s_k:=\nu_k$. Now assume $\nu_k\not\in S$, and set $s_k:=\max\{s\in S:s< \nu_k\}$. Then $\{s_k+1,s_k+2,\dotsc,\nu_k\}\cap S=\emptyset$, and we deduce that there exists a number $r_k\le s_k$ with $\sigma(r_k)/r_k \ge\sigma(\nu_k)/\nu_k$. As $s_k\in S$, we obtain $\sigma(s_k)/s_k \ge \sigma(\nu_k)/\nu_k$, implying $G(s_k) > G(\nu_k)$. Now since $\lim_{k\to\infty} \nu_k = \infty$ and $\#S=\infty$ imply $\lim_{k\to\infty} s_k = \infty$, we get $\limsup_{k\to\infty} G(s_k) \ge  e^\gamma$, as desired.
\end{proof}

\begin{proof}[Proof of SA2]
We use the following three properties of a number $s\in S$, proved by Alaoglu and Erd\H{o}s~\cite{ae}.\\

\noindent AE1. \emph{The exponents in the prime factorization of $s$ are non-increasing, that is, $s=2^{k_2} \cdot 3^{k_3} \cdot 5^{k_5} \dotsb p^{k_p}$ with $k_2\ge k_3\ge k_5\ge\dotsb\ge k_p$.}

\noindent AE2. \emph{If $q < r$ are prime factors  of $s$, then $\lvert \lfloor k_q \frac{\log q}{\log r}\rfloor - k_r\rvert \le1$.}

\noindent AE3. \emph{If $q$ is any prime factor of $s$, then $q^{k_q} < 2^{k_2+2}$.}\\

To prove SA2, fix an integer $n_0>1$. Let $K$ denote the largest exponent in the prime factorization of $n_0$, and set $P:=P(n_0)$, where $P(n)$ \emph{denotes the largest prime factor of}~$n$. As $n_0$ divides $(2 \cdot 3 \cdot 5 \dotsb P)^K$, by AE1 it suffices to show that the set
$$F:=\{s\in S: s \text{ is not divisible by }P^{K}\} = \{s\in S: 0\le k_P=k_P(s)<K\}$$
is finite.

From AE2 with $q=2$ and $r=P$, we infer that $k_2=k_2(s)$ is bounded, say $k_2(s) < B$, for all $s\in F$. Now if $q$ is any prime factor of~$s$, then AE1 implies $k_q=k_q(s) < B$, and AE3 implies  $q^{k_q} < 2^{B+2}$. The latter with $q=P(s)$ forces $P(s) < 2^{B+2}$. Therefore, $s<(2^{B+2}!)^B$ for all $s\in F$, and so $F$~is a finite set.
\end{proof}

\begin{remark}
We outline another proof of SA2. Observe first that, if $p^{k+1}$ does not divide~$n$, then (compare the proof of \cite[Theorem~329]{hw})
$$\frac{\sigma(n)}{n} \le \frac{n}{\varphi(n)} \left(1- \frac{1}{p^{k+1}}\right),$$
where $\varphi(n)$ is Euler's totient function. Together with the classical result \cite[Theorem~328]{hw}
$$\limsup_{n\to\infty} \frac{n}{\varphi(n)\log\log n} = e^\gamma,$$
this implies that there exists $\epsilon = \epsilon(n_0) > 0$ such that, if $n\gg1$ is not multiple of~$n_0$, then
$$G(n)\le e^\gamma - \epsilon,$$
so that, by SA1, $n$ cannot be SA.
\end{remark}

\section{GA numbers of the first kind} \label{SEC: GA}

Let us say that a positive integer $n$ is a \emph{GA number of the first kind} (\emph{GA1 number}) if $n$ is composite and satisfies condition~(i) in Definition~\ref{DEF: exceptional}  with $N$ replaced by $n$, that is, $G(n) \ge G(n/p)$ for all primes $p$ dividing~$n$. For example, $4$ is GA1, as are all other extraordinary numbers, if any. Also, the near miss $183783600$ is a GA number of the first kind. By Lemma~\ref{LEM: G(n/p)}, if $4\neq r\in R$, then $r$~is not a GA1 number.

Writing $p^k \parallel n$ when $p^k\mid n$ but $p^{k+1} \nmid n$, we have the following criterion for GA1 numbers.

\begin{prop} \label{PROP: GA}
A composite number $n$ is a GA number of the first kind if and only if prime $p\mid n$ implies
$$ \frac{\log \log n}{ \log \log \frac{n}{p}} \le \frac{p^{k+1}-1}{p^{k+1}-p} \qquad(p^k \parallel n).$$
\end{prop}
\begin{proof}
This follows easily from the definitions of GA1 and $G(n)$ and the formulas
\begin{equation*}
\sigma(n)= \prod_{p^k \parallel n} (1 + p + p^2 + \dotsb+p^k) = \prod_{p^k \parallel n} \frac{p^{k+1}-1}{p-1}.\qedhere
\end{equation*}
\end{proof}

The next two propositions determine all GA1 numbers with exactly two prime factors.

\begin{prop} \label{PROP: 2p}
Let $p$ be a prime. Then $2p$ is a GA number of the first kind if and only if $p=2$ or $p>5$.
\end{prop}
\begin{proof}
As $G(2)<0<G(2p)$, the number $2p$ is GA1 if and only if $G(2p)\ge G(p)$. Thus $2p$ is GA1 for $p=2$, but, by computation, not for $p=3$ and $5$. If $p>5$, then since $3\log\log x>2\log\log 2x$ for $x\ge7$, we have
$$\frac{G(2p)}{G(p)}=\frac{\sigma(2p)}{2p \log\log 2p} \div \frac{\sigma(p)}{p \log\log p} = \frac{3(p+1)}{2p\log\log 2p} \cdot \frac{p\log\log p}{p+1} = \frac{3\log\log p}{2\log\log 2p} >1.$$
Thus $2p$ is GA1 for $p=7,11,13,\dotsc$. This proves the proposition.
\end{proof}

\begin{prop} \label{PROP: pq}
Let $p\ge q$ be odd primes. Then $pq$ is not a GA1 number.
\end{prop}
\begin{proof}
As $(x+1)\log\log y < x\log\log xy$ when $x\ge y\ge3$, it follows that if $p>q\ge3$ are primes, then
$$\frac{G(pq)}{ G(q)}=\frac{(p+1)(q+1)}{pq \log\log pq} \div \frac{q+1}{q \log\log q}=\frac{(p+1)\log\log q}{p\log\log pq}<1,$$
and if $p\ge3$ is prime, then
$$\frac{G(p^2)}{G(p)}=\frac{p^2+p+1}{p^2\log\log p^2} \div \frac{p+1}{p\log\log p}= \frac{(p^2+p+1)\log\log p}{(p^2+p)\log\log p^2}<1.$$
Hence $pq$ is not GA1 for odd primes $p\ge q$.
\end{proof}

\section{Concluding remarks} \label{SEC: conclusion}

Our reformulation of the RH, like Lagarias's, is attractive because the constant $e^\gamma$ does not appear. Also, there is an elegant symmetry to the pair of conditions (i) and~(ii): the value of the function $G$ at the number $N$ is not less than its values at the quotients $N/p$ and at the multiples $aN$. The statement reformulates the Riemann Hypothesis in purely elementary terms of divisors, prime factors, multiples, and logarithms.


\begin{thebibliography}{99}

\bibitem{af} A. Akbary and Z. Friggstad, Superabundant numbers and the Riemann hypothesis, \textit{Amer. Math. Monthly} {\bf116} (2009), 273--275.

\bibitem{ae} L. Alaoglu and P. Erd\H{o}s, On highly composite and similar numbers, \textit{Trans. Amer. Math. Soc.} {\bf56} (1944), 448--469.

\bibitem{briggs} K. Briggs, Abundant numbers and the Riemann hypothesis, \textit{Experiment. Math.} {\bf15} (2006), 251--256.

\bibitem{cns} G. Caveney, J.-L. Nicolas, and J. Sondow, On SA, CA, and GA numbers, \textit{Ramanujan J.} (to appear); available at {\tt http://arxiv.org/abs/1112.6010}.

\bibitem{gronwall} T. H. Gronwall, Some asymptotic expressions in the theory of numbers, \textit{Trans. Amer. Math. Soc.} {\bf14} (1913), 113--122.

\bibitem{hw} G.~H.~Hardy and E.~M.~Wright, \textit{An Introduction to the Theory of Numbers}, D.~R.~Heath-Brown and J.~H.~Silverman, eds., 6th ed., Oxford University Press, Oxford, 2008.

\bibitem{lagarias} J. C. Lagarias, An elementary problem equivalent to the Riemann hypothesis, \textit{Amer. Math. Monthly} {\bf109} (2002), 534--543.

\bibitem{lagarias2} ------------, Problem 10949, \textit{Amer. Math. Monthly} {\bf109} (2002), 569.

\bibitem{ramanujan97} S. Ramanujan, Highly composite numbers, annotated and with a foreword by J.-L. Nicolas and G. Robin, \textit{Ramanujan J.} {\bf1} (1997), 119--153.

\bibitem{robin} G.~Robin, Grandes valeurs de la fonction somme des diviseurs et hypoth\` ese de Riemann, \textit{J. Math. Pures Appl.} {\bf63} (1984), 187--213.

\bibitem {oeis}N.~J.~A.~Sloane, The On-Line Encyclopedia of Integer Sequences, published electronically at {\tt http://oeis.org}, 2010.

 
\end{thebibliography}
\end{document}